\documentclass[10pt,tbtags]{article}
\usepackage{amssymb,amsmath,amsthm,amscd}
\usepackage{latexsym}
\usepackage{mathrsfs}
\usepackage{enumerate}
\usepackage{indentfirst}
\usepackage{fancyhdr}
\usepackage{titlesec}
\usepackage{cite}
\usepackage{color}
\usepackage{graphicx}
\usepackage{epsf}
\usepackage{lineno}
\usepackage{hyperref}
\newtheorem{thm}{Theorem}[section]

\newtheorem{lem}[thm]{Lemma}

\theoremstyle{definition}

\newtheorem{rmk}[thm]{Remark}

\title{{\Large\bf {Infinitely many homoclinic solutions for a class
of subquadratic second-order Hamiltonian systems}}
\thanks{Supported by the National Natural Science Foundation of China (NSFC) under
Grants No.11371252 and No.11501369; Research and Innovation Project of Shanghai Education Committee under Grant No.14zz120; Yangfan Program of Shanghai (14YF1409100); Chen Guang Project(14CG43) of Shanghai Municipal Education Commission, Shanghai Education Development Foundation and the Research Program of Shanghai Normal University (SK201403) and Shanghai Gaofeng Project for University Academic Program Development.}}
\author{\normalsize  Xiang Lv\thanks{E-mail:lvxiang@shnu.edu.cn.}
\\
\normalsize  {\it Department of Mathematics, Shanghai Normal
University, Shanghai 200234, PR China}\\ }
\date{}
\begin{document}
\maketitle \noindent\hrulefill \newline \noindent{\bf Abstract :}
\par In this paper, we mainly consider the existence of infinitely many homoclinic solutions for a class
of subquadratic second-order Hamiltonian systems
$\ddot{u}-L(t)u+W_u(t,u)=0$, where $L(t)$ is not necessarily positive
definite and the growth rate of potential function $W$ can be in $(1,3/2)$. Using the variant fountain theorem, we obtain the existence of infinitely many homoclinic solutions for the second-order
Hamiltonian systems.


\hrulefill
\section{Introduction and
main results}
The aim of this paper is to study the following second-order Hamiltonian systems
$$\ddot{u}-L(t)u+W_u(t,u)=0,\quad \forall\ t\in\mathbb{R} \eqno(\mbox{HS})$$
where $u=(u_1,u_2,\ldots,u_N)\in\mathbb{R}^N$, $W\in
C^1(\mathbb{R}\times\mathbb{R}^N,\mathbb{R})$ and $L\in
C(\mathbb{R},\mathbb{R}^{N\times N})$ is a symmetric matrix-valued
function. We usually say that a solution $u$ of (HS) is
homoclinic (to 0) if $u\in C^2(\mathbb{R},\mathbb{R}^N)$,
$u(t)\to 0$ and $\dot u(t)\to 0$ as $t\to \pm\infty$. Furthermore, if $u\not\equiv0$, then $u$ is called nontrivial.

\par In the applied sciences, Hamiltonian systems
can be used in many practical problems regarding gas dynamics, fluid
mechanics and celestial mechanics. It is clear
that the existence of homoclinic solutions is one of the most important problems in the
theory of Hamiltonian systems. Recently, more and more mathematicians have paid their attention to the existence and
multiplicity of homoclinic orbits for Hamiltonian systems, see [1-21].

\par For the case of that $L(t)$ and $W(t,x)$ are either independent
of $t$ or periodic in $t$, there have been several excellent results, see
\cite{AC,AR,CM,DG1,IJ,LLJ,MW,OT,OW,R}. More precisely, in the paper
\cite{R}, Rabinowitz has proved the existence of homoclinic orbits as
a limit of $2kT$-periodic solutions of (HS). Later, using the same
method, several results for general Hamiltonian systems were
obtained by Izydorek and Janczewska \cite{IJ}, Lv et al. \cite{LLJ}.

\par If $L(t)$ and $W(t,x)$ are not periodic with respect to $t$, it will become more difficult to consider the existence
of homoclinic orbits for (HS). This problem is quite different from the case mentioned above, due to the lack of compactness of the Sobolev embedding.
In \cite{RT}, Rabinowitz and Tanaka investigated system (HS) without
periodicity, both for $L$ and $W$. Specifically, they
assumed that the smallest eigenvalue of $L(t)$ tends to $+\infty$ as
$|t|\to \infty$, and showed that system (HS)
admits a homoclinic orbit by using a variant of the Mountain Pass theorem
without the Palais-Smale condition. Inspired by the work of Rabinowitz and Tanaka \cite{RT}, many results \cite{CES,D,KL,LLY,OT,OW,SCN,ZL,ZY}
were obtained for the case of aperiodicity. Most of them were presented under the following condition
that $L(t)$ is positive definite for all $t\in\mathbb{R}$,
$$(L(t)u,u)>0,\qquad \forall\ t\in\mathbb{R}\ \mbox{and}\ u\in\mathbb{R}^N\backslash\{0\}.$$
\par Motivated by \cite{D,ZL}, in this article we will study the existence of infinitely many homoclinic solutions for (HS), where
$L(t)$ is not necessarily positive definite for all $t\in\mathbb{R}$ and the growth rate of potential function $W$ can be in $(1,3/2)$. The main tool is the variant fountain theorem established in \cite{Z}. Our main results are the following theorems.

\begin{thm} Assume that $L$ and $W$ satisfy the following conditions:\\
(L1) There exists an $\alpha<1$ such that
$$l(t)|t|^{\alpha-2}\rightarrow\infty \quad as\ |t|\rightarrow\infty$$
where $l(t):=\inf\limits_{|u|=1,u\in\mathbb{R}^N}(L(t)u,u)$ is the smallest eigenvalue of $L(t)$;\\
(L2) There exist constants $\bar a>0$ and  $\bar r>0$ such that

(i) $L\in C^1(\mathbb{R},\mathbb{R}^{N\times N})$ and $|L^{'}(t)u|\leq\bar a|L(t)u|$,\quad $\forall\ |t|>\bar r$ and $u\in\mathbb{R}^N$ , or

(ii) $L\in C^2(\mathbb{R},\mathbb{R}^{N\times N})$ and $\bigl((L^{''}(t)-\bar aL(t))u,u\bigr)\leq0$,\quad $\forall\ |t|>\bar r$ and $u\in\mathbb{R}^N$ ,\\
where $L^{'}(t)=(d/dt)L(t)$ and $L^{''}(t)=(d^2/dt^2)L(t)$;\\
(W) $W(t,u)=a(t)|u|^\nu$ where $a:\mathbb{R}\rightarrow\mathbb{R}^+$ is a continuous function such that $a\in L^\mu(\mathbb{R},\mathbb{R})$,
$1<\nu<2$ is a constant, $2\leq\mu\leq\bar\nu$ and
$$\bar\nu=\left\{\begin{array}{ll}
\displaystyle\frac{2}{3-2\nu},\qquad 1<\nu<\frac32\\
\displaystyle\infty,\qquad\qquad \frac32\leq\nu<2
\end{array}\right.$$
Then (HS) possesses infinitely many homoclinic solutions.
\end{thm}

\begin{rmk} When we choose $\nu\in(1,\frac32)$, it is easy to see that $W$ satisfies the condition (W) of Theorem 1.1 but does not satisfy the corresponding conditions in
\cite{D,ZL}. Furthermore, the constant $\mu$ can be change in $[2,\bar\nu]$.
\end{rmk}

\section{Preliminaries}
In this section, for the
purpose of readability and making this paper self-contained, we will show the variational setting for (HS), which can be found in \cite{D,ZL}. In what follows, we will always assume that $L(t)$ satisfies (L1). Let $\mathcal{A}$ be the selfadjoint extension of the operator $-(d^2/dt^2)+L(t)$ with domain
 $\mathscr{D}(\mathcal{A})\subset L^2\equiv L^2(\mathbb{R},\mathbb{R}^N)$. Let us write $\{E(\lambda): -\infty<\lambda<+\infty\}$ and $|\mathcal{A}|$ for the spectral  resolution and the absolute value of $\mathcal{A}$ respectively, and denote by $|\mathcal{A}|^{1/2}$ the square root of $|\mathcal{A}|$. Define $U=I-E(0)-E(-0)$. Then $U$ commutes with
 $\mathcal{A}$, $|\mathcal{A}|$ and $|\mathcal{A}|^{1/2}$, and $\mathcal{A}=U|\mathcal{A}|$ is the polar decomposition of $\mathcal{A}$ (see \cite{K}). We write $E=\mathscr{D}(|\mathcal{A}|^{1/2})$ and introduce the following inner product on $E$
$$(u,v)_0=(|\mathcal{A}|^{1/2}u,|\mathcal{A}|^{1/2}v)_2+(u,v)_2$$
and norm
$$\|u\|_0=(u,u)_0^{1/2}.$$
Here, $(\cdot,\cdot)_2$ denotes the usual $L^2$-inner product. Therefore, $E$ is a Hilbert space. Since $C_0^\infty(\mathbb{R},\mathbb{R}^N)$ is dense
in $E$, it is obvious that $E$ is continuous embedded in $H^1(\mathbb{R},\mathbb{R}^N)$ (see \cite{D}). Furthermore, we have the following lemmas by \cite{D}.

\begin{lem}
If $L$ satisfies (L1), then $E$ is compactly embedded in $L^p\equiv L^p(\mathbb{R},\mathbb{R}^N)$ for all $1\leq p\in(2/(3-\alpha),\infty]$.
\end{lem}

\begin{lem}
Let $L$ satisfies (L1) and (L2), then $\mathscr{D}(\mathcal{A})$ is continuously embedded in $W^{2,2}(\mathbb{R},\mathbb{R}^N)$, and consequently, we have
$$|u(t)|\rightarrow0 \quad and\quad|\dot{u}(t)|\rightarrow0\quad as\ |t|\rightarrow\infty,\quad \forall\ u\in\mathscr{D}(\mathcal{A}).$$
\end{lem}

\par From \cite{D}, combining (L1) and Lemma 2.1, we can prove that $\mathcal{A}$ possesses a compact resolvent. Consequently,
the spectrum $\sigma(\mathcal{A})$ consists of eigenvalues, which can be arranged as $\lambda_1\leq\lambda_2\leq\cdots\rightarrow\infty$ (counted with multiplicity), and the corresponding system of eigenfunctions $\{e_n: n\in\mathbb{N}\}$, $\mathcal{A}e_n=\lambda_ne_n$, which forms an orthogonal basis in $L^2$. Next, we define
$$n^-=\#\{i|\lambda_i<0\},\ n^0=\#\{i|\lambda_i=0\},\ \bar n=n^-+n^0$$
and
$$E^-=\mbox{span}\{e_1,\cdots,e_{n^-}\},\ E^0=\mbox{span}\{e_{n^-+1},\cdots,e_{\bar n}\}=\mbox{Ker} \mathcal{A},\ E^+=\overline{\mbox{span}\{e_{\bar n+1},\cdots\}}.$$
Here, the closure is taken in $E$ with respect to the norm $\|\cdot\|_0$. Then
$$E=E^-\oplus E^0\oplus E^+.$$
Furthermore, we define on $E$ the following inner product
$$(u,v)=(|\mathcal{A}|^{1/2}u,|\mathcal{A}|^{1/2}v)_2+(u^0,v^0)_2,$$
and norm $$\|u\|^2=(u,u)=\||\mathcal{A}|^{1/2}u\|^2_2+\|u^0\|^2_2,$$
where $u=u^-+u^0+u^+$ and $v=v^-+v^0+v^+\in E^-\oplus E^0\oplus E^+$. It is clear that the norms $\|\cdot\|_0$ and $\|\cdot\|$ are equivalent by \cite{D}. From now on, we will take $(E,\|\cdot\|)$ instead of $(E,\|\cdot\|_0)$ as the working space without loss of generality.

\begin{rmk} We note that the decomposition $E=E^-\oplus E^0\oplus E^+$ is also orthogonal with respect to inner products $(\cdot,\cdot)$ and $(\cdot,\cdot)_2$.
Moreover, we will denote by $E=E^-\oplus E^0\oplus E^+$ the orthogonal decomposition with respect to the inner products $(\cdot,\cdot)$ unless otherwise stated.
\end{rmk}

\begin{rmk} Since the norms $\|\cdot\|_0$ and $\|\cdot\|$ are equivalent, by Lemma 2.1, for any $1\leq p\in(2/(3-\alpha),\infty]$, there exists a constant
$\beta_p>0$ such that $$\|u\|_p\leq\beta_p\|u\|,\quad \forall\ u\in E,\eqno(2.1)$$
where $\|u\|_p$ denotes the usual norm of $L^p$ and $\beta_p$ is independent of $u$.
\end{rmk}

Let $$\mathcal{O}(u,v)=(|\mathcal{A}|^{1/2}Uu,|\mathcal{A}|^{1/2}v),\quad \forall\ u,v\in E$$
be the quadratic form associated with $\mathcal{A}$, where $U$ is the
polar decomposition of $\mathcal{A}$. Given any $u\in\mathscr{D}(\mathcal{A})$ and $v\in E$, we can get
$$\mathcal{O}(u,v)=\int_\mathbb{R}\left((\dot{u},\dot{v})+(L(t)u,v)\right)dt.\eqno(2.2)$$
Note that $\mathscr{D}(\mathcal{A})$ is dense in $E$, we have (2.2) holds for all $u,v\in E$. Furthermore, by definition, it follows that
$$\mathcal{O}(u,v)=\left((P^+-P^-)u,v\right)=\|u^+\|^2-\|u^-\|^2\eqno(2.3)$$
for all $u=u^-+u^0+u^+\in E$, where $P^{\pm}:E\rightarrow E^{\pm}$ are the respective orthogonal projections.

\par Combining (2.2) and (2.3), we define the functional $\Phi$ on $E$ by
$$\begin{array}{ll}
\displaystyle\Phi(u)\!\!\!\!&=\displaystyle\frac12\int_\mathbb{R}\left(\|\dot{u}\|^2+(L(t)u,u)\right)dt-\int_\mathbb{R}W(t,u)dt\\
&=\displaystyle\frac12\|u^+\|^2-\frac12\|u^-\|^2-\int_\mathbb{R}W(t,u)dt\\
&=\displaystyle\frac12\|u^+\|^2-\frac12\|u^-\|^2-\Psi(u),
\end{array}\eqno(2.4)$$
where $\Psi(u)=\int_\mathbb{R}W(t,u)dt=\int_\mathbb{R}a(t)|u|^\nu dt$ for all $u=u^-+u^0+u^+\in E=E^-\oplus E^0\oplus E^+$.

\begin{rmk}
From (W) with Lemma 2.1, we can easily see that $\Phi$ and $\Psi$ are well defined. We will consider two cases as follows.\\
{\bf Case (i)} If $2\leq\mu<\infty$, then
$$\begin{array}{ll}
\displaystyle\left|\Psi(u)\right|\!\!\!\!&=\displaystyle\left|\int_\mathbb{R}W(t,u)dt\right|=\left|\int_\mathbb{R}a(t)|u|^\nu dt\right|\\
&\leq\displaystyle\left(\int_\mathbb{R}|a(t)|^\mu dt\right)^{\frac1\mu}\left(\int_\mathbb{R}|u|^{\nu\mu^*}dt\right)^{\frac{1}{\mu^*}}\\
&=\displaystyle\|a\|_\mu\|u\|_{\nu\mu^*}^\nu<\infty
\end{array}$$
where $\frac1\mu+\frac{1}{\mu^*}=1$, $\nu\mu^*\geq1$.\\
{\bf Case (ii)} If $\mu=\infty$, then $|\Psi(u)|\leq\|a\|_\infty\|u\|_\nu^\nu<\infty$.
\end{rmk}

\begin{lem}
Let (L1), (L2) and (W) hold. Then $\Psi\in C^1(E,\mathbb{R})$ and $\Psi': E\rightarrow E^*$ is compact, and consequently $\Phi\in C^1(E,\mathbb{R})$.
Moreover,
$$\Psi'(u)v=\int_\mathbb{R}\left(W_u(t,u),u\right)dt=\int_\mathbb{R}\left(\nu a(t)|u|^{\nu-2}u,v\right)dt\eqno(2.5)$$
$$\begin{array}{ll}
\displaystyle\Phi'(u)v\!\!\!\!&=\displaystyle(u^+,v^+)-(u^-,v^-)-\Psi'(u)v\\
&=\displaystyle(u^+,v^+)-(u^-,v^-)-\int_\mathbb{R}\left(W_u(t,u),v\right)dt
\end{array}\eqno(2.6)$$
for all $u=u^-+u^0+u^+$ and $v=v^-+v^0+v^+\in E^-\oplus E^0\oplus E^+$. Moreover, all critical points of $\Phi$ on $E$ are homoclinic solutions of (HS) satisfying $u(t)\to 0$ and $\dot u(t)\to 0$ as $|t|\to\infty$.
\end{lem}

\begin{proof}We first show that (2.5) holds by definition. If $2\leq\mu<\infty$, then $1<\mu^*\leq2$, where $\frac1\mu+\frac{1}{\mu^*}=1$. For any given $u, v\in E$, by the Mean Value Theorem and the H\"older inequality, we have
$$\begin{array}{ll}
\displaystyle&\quad\displaystyle\left|\int_\mathbb{R}\left[W(t,u+v)-W(t,u)-\bigl(W_u(t,u),v\bigr)\right]dt\right|\\
&=\displaystyle\left|\int_\mathbb{R}\left[\int_0^1\bigl(W_u(t,u+\theta v)-W_u(t,u),v\bigr)d\theta\right]dt\right|\\
&\leq\displaystyle2\nu\int_\mathbb{R}|a(t)|(|u|+|v|)^{\nu-1}|v|dt\\
&\leq\displaystyle2\nu\int_\mathbb{R}|a(t)|(|u|^{\nu-1}+|v|^{\nu-1})|v|dt\\
&\leq\displaystyle2\nu\left(\int_\mathbb{R}|a(t)|^\mu dt\right)^{\frac1\mu}\left(\int_\mathbb{R}|u|^{\mu^*(\nu-1)}|v|^{\mu^*}dt\right)^{\frac{1}{\mu^*}}\\
&+\displaystyle2\nu\left(\int_\mathbb{R}|a(t)|^\mu dt\right)^{\frac1\mu}\left(\int_\mathbb{R}|v|^{\mu^*\nu}dt\right)^{\frac{1}{\mu^*}}\\
&\leq\displaystyle2\nu\|a\|_\mu\left(\int_\mathbb{R}|u|^2dt\right)^{\frac{\nu-1}{2}}\left(\int_\mathbb{R}|v|^\frac{2\mu^*}{2+\mu^*-\mu^*\nu}dt\right)^{\frac{2+\mu^*-\mu^*\nu}{2\mu^*}}
+2\nu\|a\|_\mu\|v\|_{\mu^*\nu}^{\nu}\\
&=\displaystyle2\nu\|a\|_\mu\|u\|_2^{\nu-1}\|v\|_{\frac{2\mu^*}{2+\mu^*-\mu^*\nu}}+2\nu\|a\|_\mu\|v\|_{\mu^*\nu}^{\nu}\\
&\leq\displaystyle2\nu\beta_{\frac{2\mu^*}{2+\mu^*-\mu^*\nu}}\|a\|_\mu\|u\|_2^{\nu-1}\|v\|+2\nu\beta_{\mu^*\nu}^{\nu}\|a\|_\mu\|v\|^{\nu}\rightarrow0,\qquad \mbox{as $v\rightarrow0$ in $E$}
\end{array}\eqno(2.7)$$
where $\frac{2\mu^*}{2+\mu^*-\mu^*\nu}\geq1$ and the second inequality holds by the fact that if $0<p<1$, then $(|a|+|b|)^p\leq|a|^p+|b|^p$, $\forall\ a, b\in\mathbb{R}$. If $\mu=\infty$, then similar to the proof of (2.7), we can obtain
$$\begin{array}{ll}
\displaystyle&\quad\displaystyle\left|\int_\mathbb{R}\left[W(t,u+v)-W(t,u)-\bigl(W_u(t,u),v\bigr)\right]dt\right|\\
&\leq\displaystyle2\nu\|a\|_\infty(\|u\|_\infty^{\nu-1}+\|v\|_\infty^{\nu-1})\int_\mathbb{R}|v|dt\\
&\leq\displaystyle2\nu\|a\|_\infty\beta_\infty^{\nu-1}\beta_1(\|u\|^{\nu-1}+\|v\|^{\nu-1})\|v\|
\rightarrow0,\qquad \mbox{as $v\rightarrow0$ in $E$}
\end{array}\eqno(2.8)
$$
where the last inequality holds by (2.1) and $\beta_\infty, \beta_1$ are constants there. Combining (2.7) and (2.8), (2.5) holds immediately by the definition of Fr\'{e}chet derivatives. Consequently, (2.6) also holds due to the definition of $\Phi$.

Next, we verify that $\Psi': E\rightarrow E^*$ is compact. Let $u_n\rightharpoonup u_0$ (weakly) in $E$, by Lemma 2.1, we have $u_n\rightarrow u_0$
in $L^p$ for all $1\leq p\in(2/(3-\alpha),\infty]$. If $2\leq\mu<\infty$, using the H\"older inequality, we can obtain
$$\begin{array}{ll}
\displaystyle\|\Psi'(u_n)-\Psi'(u_0)\|_{E^*}\!\!\!\!&=\displaystyle\sup_{\|v\|=1}\|(\Psi'(u_n)-\Psi'(u_0))v\|\\
&=\displaystyle\sup_{\|v\|=1}\left|\int_\mathbb{R}\bigl(W_u(t,u_n)-W_u(t,u_0),v\bigr)dt\right|\\
&\leq\displaystyle\sup_{\|v\|=1}\left[\left(\int_\mathbb{R}|W_u(t,u_n)-W_u(t,u_0)|^\mu dt\right)^{\frac1\mu}\|v\|_{\mu^*}\right]\\
&\leq\displaystyle\beta_{\mu^*}\left(\int_\mathbb{R}|W_u(t,u_n)-W_u(t,u_0)|^\mu dt\right)^{\frac1\mu},\qquad \forall\ n\in\mathbb{N}
\end{array}\eqno(2.9)
$$
where the last inequality holds by (2.1) and $\beta_\mu^*$ is the constant there, $\frac1\mu+\frac{1}{\mu^*}=1$. Next, we will prove that
$W_u(t,u_n)\rightarrow W_u(t,u_0)$ in $L^\mu(\mathbb{R},\mathbb{R}^N)$. Observing that ${u_n}$ is bounded in $L^\infty$, then by the Jensen inequality, we have
$$\begin{array}{ll}
\displaystyle&\quad\displaystyle\int_\mathbb{R}|W_u(t,u_n)-W_u(t,u_0)|^\mu dt\\
&\leq\displaystyle2^{\mu-1}\nu^\mu\int_\mathbb{R}|a(t)|^\mu(|u_n|^\mu+|u_0|^\mu)dt\\
&\leq\displaystyle2^{\mu-1}\nu^\mu\int_\mathbb{R}|a(t)|^\mu(\|u_n\|_\infty^\mu+\|u_0\|_\infty^\mu)dt\\
&\leq\displaystyle2^{\mu-1}\nu^\mu M\int_\mathbb{R}|a(t)|^\mu dt\\
\end{array}
$$
where $M=2\max\{\|u_0\|_\infty^\mu, \|u_n\|_\infty^\mu, \forall\ n\in\mathbb{N}\}$. Combining the fact that $u_n\rightarrow u_0$ in $L^\infty$ and the Lebesgue's Dominated Convergence Theorem,
$$\left(\int_\mathbb{R}|W_u(t,u_n)-W_u(t,u_0)|^\mu dt\right)^{\frac1\mu}\rightarrow0, \qquad \mbox{as}\qquad n\rightarrow\infty.$$
Next, we will deal with the case of $\mu=\infty$ (i.e. $\nu>\frac32$), this part is mainly motivated by the proof of Lemma 2 in \cite{OW}. By the H\"older inequality, we have
$$\begin{array}{ll}
\displaystyle\|\Psi'(u_n)-\Psi'(u_0)\|_{E^*}\!\!\!\!&\leq\displaystyle
\sup_{\|v\|=1}\left[\left(\int_\mathbb{R}|W_u(t,u_n)-W_u(t,u_0)|^2dt\right)^{\frac12}\|v\|_2\right]\\
&\leq\displaystyle\beta_2\left(\int_\mathbb{R}|W_u(t,u_n)-W_u(t,u_0)|^2dt\right)^{\frac12},\qquad \forall\ n\in\mathbb{N}
\end{array}\eqno(2.10)
$$
We note that by Lemma 2.1, $u_n\rightarrow u_0$ in $L^{2(\nu-1)}$ for $\nu>\frac32$, passing to a subsequence if necessary, it can be assumed that
$$\sum_{n=1}^{\infty}\|u_n-u_0\|_{2(\nu-1)}<+\infty,$$
which implies that $$\sum_{n=1}^{\infty}|u_n(t)-u_0(t)|=g(t)\in L^{2(\nu-1)}(\mathbb{R},\mathbb{R}).$$
Since $\nu>\frac32$, then
$$\begin{array}{ll}
\displaystyle&\quad\displaystyle\int_\mathbb{R}|W_u(t,u_n)-W_u(t,u_0)|^2dt\\
&\leq\displaystyle\int_\mathbb{R}2\nu^2|a(t)|^2(|u_n|^{2(\nu-1)}+|u_0|^{2(\nu-1)})dt\\
&\leq\displaystyle\int_\mathbb{R}2\nu^2|a(t)|^2(2^{2\nu-3}|u_n-u_0|^{2(\nu-1)}+(2^{2\nu-3}+1)|u_0|^{2(\nu-1)})dt\\
&\leq\displaystyle2^{2\nu-1}\nu^2\|a\|_\infty^2\int_\mathbb{R}(|g(t)|^{2(\nu-1)}+|u_0|^{2(\nu-1)})dt\\
&\leq\displaystyle2^{2\nu-1}\nu^2\|a\|_\infty^2(\|g\|_{2(\nu-1)}^{2(\nu-1)}+\beta_{2(\nu-1)}^{2(\nu-1)}\|u_0\|^{2(\nu-1)})
\end{array}
$$
Applying the Lebesgue's Dominated Convergence Theorem, we have
$$\left(\int_\mathbb{R}|W_u(t,u_n)-W_u(t,u_0)|^2dt\right)^{\frac12}\rightarrow0, \qquad \mbox{as}\qquad n\rightarrow\infty.$$
Consequently, $\Psi'$ is weakly continuous, and so $\Psi'$ is continuous. Therefore $\Psi\in C^1(E,\mathbb{R})$ and hence $\Phi\in C^1(E,\mathbb{R})$.
Moreover, $\Psi'$ is compact due to the weak continuity of $\Psi'$ and the fact that $E$ is a Hilbert Space.

Finally, we will prove that all critical points of $\Phi$ on $E$ are homoclinic solutions of (HS). By the standard procedure, we can see that any critical points of $\Phi$ on $E$ satisfy (HS) and $u\in C^2(\mathbb{R},\mathbb{R}^N)$.  We note that if $1<\nu<\frac32$, then $2\leq\mu\leq\frac{2}{3-2\nu}$. For $\mu=2$, by (HS), we have
$$\begin{array}{ll}
\displaystyle\|\mathcal{A}u\|_2^2\!\!\!\!&=\displaystyle\int_{\mathbb{R}}|W_u(t,u)|^2dt\\
&\leq\displaystyle\nu^2\|u\|_\infty^{2(\nu-1)}\int_{\mathbb{R}}|a(t)|^2dt\\
&\leq\displaystyle\nu^2\beta_\infty^{2(\nu-1)}\|u\|^{2(\nu-1)}\int_{\mathbb{R}}|a(t)|^\mu dt<\infty.
\end{array}\eqno(2.11)$$
In the case of $2<\mu\leq\frac{2}{3-2\nu}$, then
$$\begin{array}{ll}
\displaystyle\|\mathcal{A}u\|_2^2\!\!\!\!&=\displaystyle\int_{\mathbb{R}}|W_u(t,u)|^2dt\\
&\leq\displaystyle\nu^2\left(\int_{\mathbb{R}}|a(t)|^\mu dt\right)^{\frac2\mu}\left(\int_{\mathbb{R}}|u|^{2\bar\mu(\nu-1)} dt\right)^{\frac{1}{\bar\mu}}\\
&\leq\displaystyle\nu^2\|u\|_{2\bar\mu(\nu-1)}^{2(\nu-1)}\left(\int_{\mathbb{R}}|a(t)|^\mu dt\right)^{\frac2\mu}\\
&\leq\displaystyle\nu^2\beta_{2\bar\mu(\nu-1)}^{2(\nu-1)}\|u\|^{2(\nu-1)}\left(\int_{\mathbb{R}}|a(t)|^\mu dt\right)^{\frac2\mu}<\infty,
\end{array}\eqno(2.12)$$
where $\frac2\mu+\frac{1}{\bar\mu}=1$ and $2\bar\mu(\nu-1)\geq1$ because of $\mu\leq\frac{2}{3-2\nu}$. If $\frac32\leq\nu<2$, combining the fact that $2(\nu-1)\geq1$ and H\"older inequality, similar to the proof of (2.11) and (2.12), we can get the same result. Consequently, $u\in\mathscr{D}(\mathcal{A})$ and hence $u$ is a homoclinic solution of (HS) by Lemma 2.2. The proof is complete.
\end{proof}

In the next argument, the following variant fountain theorem will be used to prove our main results. Let $E$ be a Banach space with the norm $\|\cdot\|$ and $E=\overline{\bigoplus_{j\in\mathbb{N}}X_j}$ with dim$X_j<\infty$ for any $j\in\mathbb{N}$. We write $Y_k=\bigoplus_{j=1}^kX_j$ and $Z_k=\overline{\bigoplus_{j=k}X_j}$. The $C^1$-functional $\Phi_\lambda:E\rightarrow\mathbb{R}$ is given by
$$\Phi_\lambda(u):=A(u)-\lambda B(u),\qquad \lambda\in[1,2].$$
\begin{thm} ([22, Theorem2.2.]) Assume that the functional $\Phi_\lambda$ defined above satisfies\\
(F1) $\Phi_\lambda$ maps bounded sets to bounded sets uniformly for $\lambda\in[1,2]$. Furthermore, $\Phi_\lambda(-u)=\Phi_\lambda(u)$ for all
\par\quad $(\lambda,u)\in[1,2]\times E$;\\
(F2) $B(u)\geq0$; $B(u)\rightarrow\infty$ as $\|u\|\rightarrow\infty$ on any finite dimensional subspace of $E$;\\
(F3) There exist $\rho_k>r_k>0$ such that
$$a_k(\lambda):=\inf_{u\in Z_k,\|u\|=\rho_k}\Phi_\lambda(u)\geq0>b_k(\lambda):=\max_{u\in Y_k,\|u\|=r_k}\Phi_\lambda(u),\quad \forall\ \lambda\in[1,2]$$
and
$$d_k(\lambda):=\inf_{u\in Z_k,\|u\|\leq\rho_k}\Phi_\lambda(u)\rightarrow0\quad as\ k\rightarrow\infty\ uniformly\ for\ \lambda\in[1,2].$$
Then there exist $\lambda_n\rightarrow1$, $u_{\lambda_n}\in Y_n$ such that
$$\Phi'_{\lambda_n}|_{Y_n}(u_{\lambda_n})=0,\ \Phi_{\lambda_n}(u_{\lambda_n})\rightarrow c_k\in[d_k(2),b_k(1)]\quad as\quad n\rightarrow\infty.$$
In particular, if $\{u_{\lambda_n}\}$ has a convergent subsequence for every $k$, then $\Phi_1$ has infinitely many nontrivial critical points
$\{u_k\}\in E\backslash\{0\}$ satisfying $\Phi_1(u_k)\rightarrow0^-$ as $k\rightarrow\infty$.
\end{thm}
In order to make use of Theorem 2.7, we consider the functionals $A$, $B$ and $\Phi_\lambda$ on the working space defined $E=\mathscr{D}(|\mathcal{A}|^{1/2})$ by
$$A(u)=\frac12\|u^+\|^2,\qquad B(u)=\frac12\|u^-\|^2+\int_\mathbb{R}W(t,u)dt,\eqno(2.13)$$
and
$$\Phi_\lambda(u)=A(u)-\lambda B(u)=\frac12\|u^+\|^2-\lambda\left(\frac12\|u^-\|^2+\int_\mathbb{R}W(t,u)dt\right)\eqno(2.14)$$
for all $u=u^-+u^0+u^+\in E$ and $\lambda\in[1,2]$. By Lemma 2.6, it is clear that $\Phi_\lambda\in C^1(E,\mathbb{R})$ for all $\lambda\in[1,2]$. Let $X_j:=\mathbb{R}e_j=$span$\{e_j\}$, $j\in\mathbb{N}$, where $\{e_j, j\in\mathbb{N}\}$ is the system of eigenfunctions and the orthogonal basis in $L^2$ below Lemma 2.2. Furthermore, it is evident that $\Phi_1=\Phi$, where $\Phi$ is the functional defined in (2.4).

\section{Proof of theorems}

\begin{lem}
Let (L1), (L2) and (W) hold, then $B(u)\geq0$. Moreover, $B(u)\rightarrow\infty$ as $\|u\|\rightarrow\infty$ on any finite dimensional subspace of $E$.
\end{lem}

\begin{proof} By definitions of the functional $B$ and (W), $B(u)\geq0$ holds obviously. Next we will prove that $B(u)\rightarrow\infty$ as $\|u\|\rightarrow\infty$ on any finite dimensional subspace of $E$. First we claim that for any finite dimensional subspace $F\subset E$, there exists $\varepsilon>0$ such that
$$meas\{t\in\mathbb{R}: a(t)|u(t)|^\nu\geq\varepsilon\|u\|^\nu\}\geq\varepsilon,\qquad \forall\ u\in F\backslash\{0\}.\eqno(3.1)$$
The proof of (3.1) is very similar as that of \cite{SCN}. We omit it here. Now, let
$$\Omega_u=\{t\in\mathbb{R}: a(t)|u(t)|^\nu\geq\varepsilon\|u\|^\nu\},\qquad \forall\ u\in F\backslash\{0\},\eqno(3.2)$$
where $\varepsilon$ is given in (3.1). From (3.1), we can obtain that
$$meas(\Omega_u)\geq\varepsilon,\qquad \forall\ u\in F\backslash\{0\},\eqno(3.3)$$
Combining (W) and (3.3), for all $u\in F\backslash\{0\}$, we can see that
$$\begin{array}{ll}
\displaystyle B(u)\!\!\!\!&=\displaystyle\frac12\|u^-\|^2+\int_\mathbb{R}W(t,u)dt\\
&\geq\displaystyle\int_{\Omega_u}a(t)|u(t)|^\nu dt\\
&\geq\varepsilon\|u\|^\nu meas(\Omega_u)\geq\varepsilon^2\|u\|^\nu.
\end{array}\eqno(3.4)$$
This implies $B(u)\rightarrow\infty$ as $\|u\|\rightarrow\infty$ on any finite dimensional subspace of $E$. If $\mu=\infty$, similar to the case of $2\leq\mu<\infty$, by the standard procedure, we can prove that there exists $\varepsilon_1>0$ such that
$$meas\{t\in\mathbb{R}: a(t)|u(t)|^\nu\geq\varepsilon_1\|u\|^\nu\}\geq\varepsilon_1,\qquad \forall\ u\in F\backslash\{0\}.\eqno(3.5)$$
Therefore, by (3.4), we can conclude that $B(u)\rightarrow\infty$ as $\|u\|\rightarrow\infty$ on any finite dimensional subspace of $E$. The proof is complete.
\end{proof}

\begin{lem}
Under the conditions in Theorem 1.1, then there exists a sequence $\rho_k\rightarrow0^+$ as $k\rightarrow\infty$ such that
$$a_k(\lambda):=\inf_{u\in Z_k,\|u\|=\rho_k}\Phi_\lambda(u)\geq0,\quad \forall\ \lambda\in[1,2],\ k\geq \bar n+1,$$
and
$$d_k(\lambda):=\inf_{u\in Z_k,\|u\|\leq\rho_k}\Phi_\lambda(u)\rightarrow0\quad as\ k\rightarrow\infty\ uniformly\ for\ \lambda\in[1,2].$$
where $Z_k=\overline{\bigoplus_{j=k}X_j}$ for all $k\in\mathbb{N}$.
\end{lem}

\begin{proof} By the definition of $\bar n$ below the Lemma 2.2, we can know that $Z_k\subset E^+$ for all $k\geq\bar n+1$. Therefore, for all $k\geq\bar n+1$, from (W) and (2.14), it follows that
$$\begin{array}{ll}
\displaystyle\Phi_\lambda(u)\!\!\!\!&=\displaystyle\frac12\|u\|^2-\lambda\int_\mathbb{R}W(t,u)dt\\
[-1ex]\\[-1ex]
&\geq\displaystyle\frac12\|u\|^2-2\int_\mathbb{R}W(t,u)dt\\
[-1ex]\\[-1ex]
&=\displaystyle\frac12\|u\|^2-2\int_\mathbb{R}a(t)|u|^\nu dt,\qquad \forall\ (\lambda,u)\in[1,2]\times Z_k.
\end{array}\eqno(3.6)$$
If $2\leq\mu<\infty$, let $\eta_k:=\sup\limits_{u\in Z_k,\|u\|=1}\|u\|_{\nu\mu^*}$, where $\frac1\mu+\frac{1}{\mu^*}=1$. By Lemma 2.1, we can conclude that $\eta_k\rightarrow0$ as $k\rightarrow\infty$. Therefore, combining (3.6) with (W), we have
$$\Phi_\lambda(u)\geq\frac12\|u\|^2-2\|a\|_\mu\|u\|_{\nu\mu^*}^\nu\geq\frac12\|u\|^2-2\eta_k^\nu\|a\|_\mu\|u\|^\nu,\qquad \forall\ (\lambda,u)\in[1,2]\times Z_k.\eqno(3.7)$$
Let $\rho_k:=(8\eta_k^\nu\|a\|_\mu)^{1/(2-\nu)}$, the rest of proof is very similar as that of \cite{SCN}. We omit it here.
For the case of $\mu=\infty$, similar to the above procedure, the same result can be obtained. We omit it here.
The proof is complete.
\end{proof}

\begin{lem}Assume that (L1), (L2) and (W) hold, then for the sequence $\{\rho_k\}_{k\in\mathbb{N}}$ obtained in Lemma 3.2, there exists a
sequence $\{r_k\}_{k\in\mathbb{N}}$ such that $\rho_k>r_k>0$ for $\forall\ k\in\mathbb{N}$ and
$$b_k(\lambda):=\max_{u\in Y_k,\|u\|=r_k}\Phi_\lambda(u)<0,\quad \forall\ \lambda\in[1,2].\eqno(3.8)$$
where $Y_k=\bigoplus_{j=1}^kX_j=\mbox{span}\{e_1,\ldots,e_k\}$ for $\forall\ k\in\mathbb{N}$.
\end{lem}

\begin{proof}For $\forall\ k\in\mathbb{N}$, it is clear that $Y_k$ is a finite dimensional subspace of $E$. Therefore, for $\forall\ \lambda\in[1,2]$, from
(W), (3.2), (3.3) and (3.5), let $\varepsilon_0=\min\{\varepsilon,\varepsilon_1\}$, we have
$$\begin{array}{ll}
\displaystyle\Phi_\lambda(u)\!\!\!\!&=\displaystyle\frac12\|u^+\|^2-\lambda\left(\frac12\|u^-\|^2+\int_\mathbb{R}W(t,u)dt\right)\\
&\leq\displaystyle\frac12\|u\|^2-\int_\mathbb{R}W(t,u)dt\\
&\leq\displaystyle\frac12\|u\|^2-\int_{\Omega_u}a(t)|u|^\nu dt\\
&\leq\displaystyle\frac12\|u\|^2-\varepsilon_0\|u\|^\nu\mbox{meas}(\Omega_u)\\
&\leq\displaystyle\frac12\|u\|^2-\varepsilon_0^2\|u\|^\nu,\qquad\qquad\qquad\qquad\forall\ u\in Y_k,  k\in\mathbb{N}.
\end{array}\eqno(3.9)$$
For $\forall\ k\in\mathbb{N}$, we choose $0<r_k<\min\{\rho_k,\varepsilon_0^{\frac{2}{2-\nu}}\}$. From (3.9), an easy computation shows that
$$b_k(\lambda):=\max_{u\in Y_k,\|u\|=r_k}\Phi_\lambda(u)\leq-\frac{r_k^2}{2}<0,\qquad \forall\ k\in\mathbb{N}.$$
The proof is complete.
\end{proof}

\par Next we will present the proof of our main result.\\  \\
{\bf Proof of Theorem 1.1.} Combining Remark 2.5 and (2.14), it is clear that the condition (F1) in Theorem 2.7 holds obviously.
By Lemma 3.1, 3.2 and 3.3, we can easily see that conditions (F2) and (F3) in Theorem 2.7 hold for all $k\geq \bar n+1$. Consequently,
from Theorem 2.7, for all $k\geq \bar n+1$, there exist $\lambda_n\rightarrow1$, $u_{\lambda_n}\in Y_n$ such that
$$\Phi'_{\lambda_n}|_{Y_n}(u_{\lambda_n})=0,\ \Phi_{\lambda_n}(u_{\lambda_n})\rightarrow c_k\in[d_k(2),b_k(1)]\quad\mbox{as}\  n\rightarrow\infty.\eqno(3.10)$$
\par In what follows, the fist step is to show that $\{u_{\lambda_n}\}$ is bounded in $E$. For the case of $2\leq\mu<\infty$, since
$\Phi'_{\lambda_n}|_{Y_n}(u_{\lambda_n})=0$, by (2.6) and (2.14), we have
$$\Phi'_{\lambda_n}|_{Y_n}(u_{\lambda_n})u_{\lambda_n}^+=\|u_{\lambda_n}^+\|^2-\lambda_n\int_\mathbb{R}\left(W_u(t,u_{\lambda_n}),u_{\lambda_n}^+\right)dt=0.\eqno(3.11)$$
Therefore, using (W) and the H\"older inequality, it follows that
$$\begin{array}{ll}
\displaystyle\|u_{\lambda_n}^+\|^2\!\!\!\!&=\displaystyle\lambda_n\int_\mathbb{R}\left(W_u(t,u_{\lambda_n}),u_{\lambda_n}^+\right)dt\\
&\leq\displaystyle2\int_\mathbb{R}|a(t)||u_{\lambda_n}|^{\nu-1}|u_{\lambda_n}^+|dt\\
&\leq\displaystyle2\left(\int_\mathbb{R}|a(t)|^\mu dt\right)^{\frac1\mu}\left(\int_\mathbb{R}|u_{\lambda_n}|^{\mu^*(\nu-1)}|u_{\lambda_n}^+|^{\mu^*}dt\right)^{\frac{1}{\mu^*}}\\
&\leq\displaystyle2\nu\|a\|_\mu\left(\int_\mathbb{R}|u_{\lambda_n}|^2dt\right)^{\frac{\nu-1}{2}}\left(\int_\mathbb{R}|u_{\lambda_n}^+|^\frac{2\mu^*}{2+\mu^*-\mu^*\nu}dt\right)^{\frac{2+\mu^*-\mu^*\nu}{2\mu^*}}
\\
&=\displaystyle2\nu\|a\|_\mu\|u_{\lambda_n}\|_2^{\nu-1}\|u_{\lambda_n}^+\|_{\frac{2\mu^*}{2+\mu^*-\mu^*\nu}}\\
&\leq\displaystyle M_1\|a\|_\mu\|u_{\lambda_n}\|^\nu
\end{array}\eqno(3.12)$$
for some $M_1>0$, where $\frac1\mu+\frac{1}{\mu^*}=1$, $\frac{2\mu^*}{2+\mu^*-\mu^*\nu}\geq1$ and the last inequality holds because of (2.1).  Furthermore, combing (2.6) with (3.10) and the H\"older inequality, we have
$$\begin{array}{ll}
\displaystyle-\Phi_{\lambda_n}(u_{\lambda_n})\!\!\!\!&=\displaystyle\frac12\Phi'_{\lambda_n}|_{Y_n}(u_{\lambda_n})u_{\lambda_n}-\Phi_{\lambda_n}(u_{\lambda_n})\\
&=\displaystyle\lambda_n(1-\frac\nu2)\int_\mathbb{R}|a(t)||u_{\lambda_n}|^\nu dt\\
&\geq\displaystyle\frac{1}{2^{\nu-1}}\lambda_n(1-\frac\nu2)\int_\mathbb{R}|a(t)||u_{\lambda_n}^-+u_{\lambda_n}^0|^\nu dt-\lambda_n(1-\frac\nu2)\int_\mathbb{R}|a(t)||u_{\lambda_n}^+|^\nu dt\\
&\geq\displaystyle\frac{\varepsilon^2}{2^{\nu-1}}\lambda_n(1-\frac\nu2)\|u_{\lambda_n}^-+u_{\lambda_n}^0\|^\nu-\lambda_n(1-\frac\nu2)\|a\|_\mu\|u_{\lambda_n}^+\|_{\mu^*\nu}^\nu
\end{array}\eqno(3.13)$$
where the last inequality holds by the fact that dim$(E^-\oplus E^0)<\infty$ and (3.1). Note that $1<\nu<2$, then (3.12) and (3.13) implies that $\{\|u_{\lambda_n}^+\|\}$ is bounded. Next, we just have to show that $\{\|u_{\lambda_n}^-+u_{\lambda_n}^0\|\}$ is also bounded. Consequently, from (3.13) and (2.1), we get
$$\|u_{\lambda_n}^-+u_{\lambda_n}^0\|^\nu\leq-M_2\Phi_{\lambda_n}(u_{\lambda_n})+M_3\|u_{\lambda_n}^+
\|_{\mu^*\nu}^\nu\leq-M_2\Phi_{\lambda_n}(u_{\lambda_n})+M_4\|u_{\lambda_n}^+
\|^\nu\eqno(3.14)$$
for some positive constants $M_2$, $M_3$ and $M_4$. Notice that $\{\|u_{\lambda_n}^+\|\}$ is bounded, by (3.10), we can conclude that $\{\|u_{\lambda_n}^-+u_{\lambda_n}^0\|\}$ is also bounded. Therefore, there exists $M_5>0$ such that
$\|u_{\lambda_n}\|^2=\|u_{\lambda_n}^+\|^2+\|u_{\lambda_n}^-+u_{\lambda_n}^0\|^2\leq M_5$, i.e. $\{u_{\lambda_n}\}$ is bounded in $E$.

Finally, we prove that $\{u_{\lambda_n}\}$ has a strong convergent subsequence in $E$. The proof of this assertion can be accomplished as that of \cite{SCN}. We omit it here.

Now by the last conclusion of Theorem 2.7, we obtain that $\Phi=\Phi_1$ has infinitely many nontrivial critical points. Consequently, (HS) possesses infinitely many homoclinic solutions by Lemma 2.6. The proof of Theorem 1.1 is complete.\quad\quad$\Box$
\begin{rmk}
In this paper, we have considered the existence of infinitely many homoclinic solutions for a class of subquadratic
second-order Hamiltonian systems, where $1<\nu<\frac32$ is allowed. We view this result as merely one first step in the theory for the case of $1<\nu<\frac32$, there are still many problems to pursue. For example, when $1<\nu<\frac32$, the upper bound of $\mu$ whether can be $\infty$, what we will discuss in the future study.
\end{rmk}
\noindent{\bf Acknowledgements}
\par The authors would like to thank the reviewer for their
valuable comments and suggestions.\\

\end{document}